\documentclass[11pt]{article}
\usepackage{eurosym}
\usepackage{amsmath,amssymb,amsthm,amsfonts}
\usepackage{subcaption}
\usepackage{appendix}
\usepackage{amsmath}
\usepackage{mathrsfs}
\usepackage{amsfonts}
\usepackage{amssymb}
\usepackage{multirow}
\usepackage{rotating}
\usepackage{xcolor}
\usepackage{lscape}
\usepackage{verbatim}
\usepackage{graphicx}
\usepackage{cancel}
\usepackage[normalem]{ulem}
\usepackage{url}
\usepackage{booktabs}
\usepackage{natbib}
\usepackage{apalike}
\usepackage{algorithm}
\usepackage{algpseudocode}

\numberwithin{equation}{section} 

\newtheorem{theorem}{Theorem}[section]
\newtheorem{corollary}{Corollary}[section]
\newtheorem{lemma}{Lemma}[section]
\newtheorem{proposition}{Proposition}[section]
\newtheorem{definition}{Definition}[section]

\newtheorem{example}{Example}
 
\usepackage{authblk}
\usepackage{tikz}

\newcommand{\xx}{\boldsymbol{x}}
\newcommand{\E}{\mathbb{E}}
\newcommand{\KK}{\boldsymbol{K}}
\newcommand{\PP}{\mathbb{P}}
\newcommand{\ii}{\boldsymbol{i}}
\newcommand{\II}{\boldsymbol{I}}

\newcommand{\HHH}{\mathcal{H}}
\newcommand{\GGG}{\mathcal{G}}
\newcommand{\EEE}{\mathcal{E}}
\newcommand{\ee}{\boldsymbol{e}}
\newcommand{\PPP}{\mathcal{P}}
\newcommand{\pp}{\boldsymbol{p}}
\newcommand{\kk}{\boldsymbol{k}}
\newcommand{\NN}{\mathbb{N}}
\newcommand{\RR}{\mathbb{R}}
\newcommand{\BBB}{\mathcal{B}}
\newcommand{\zz}{\boldsymbol{z}}
\newcommand{\JJ}{\boldsymbol{J}}

\begin{document}

\title{Conway--Maxwell multivariate Bernoulli distribution}
\author[1]{Hélène Cossette}
\author[1]{Etienne Marceau}
\author[2]{Alessandro Mutti}
\author[2]{Patrizia Semeraro}
	\affil[1]{\textit{\'Ecole d'actuariat, Université Laval}}
\affil[2]{\textit{Department of Mathematical Sciences, Politecnico di Torino}}
\maketitle

\begin{abstract}
    
We investigate the Conway--Maxwell multivariate Bernoulli distributions, a family of multivariate Bernoulli distributions derived from the Conway--Maxwell-binomial distribution. 
We show that it is possible to set the parametrization such that the Bernoulli marginals remain intact, allowing us to study dependence properties within this family. 
In particular, we demonstrate that this family spans the full spectrum of dependence. 
Moreover, for specific ranges of the parameters, these distributions satisfy the strongly Rayleigh property, a negative dependence notion stronger than negative association.
\end{abstract}


\section{Introduction}

In this paper, we study a two-parameter class of exchangeable $d$-variate Bernoulli distributions and we demonstrate that this family spans the full spectrum of dependence. Moreover, for specific ranges of the parameters, these distributions satisfy a strong negative dependence property --- namely, the strongly Rayleigh property (SR). Since the SR property implies negative association and can often be verified directly from the probability mass function (pmf), it provides a convenient and powerful tool for establishing negative dependence in this setting. 
We refer to this class as the Conway–Maxwell multivariate Bernoulli distributions, denoted by $\text{CMMB}_d(r, \nu)$, since it is the family of exchangeable Bernoulli distributions underlying the Conway–Maxwell-binomial distribution $\text{CMB}_d(r, \nu)$ introduced in \cite{kadane2013number}. 
Although any discrete distribution with support on $\{0,\ldots, d\}$ can be represented in many different ways as the distribution of the sum of Bernoulli random variables, it admits a unique representation as the sum of exchangeable Bernoulli random variables (\cite{fontana2021model}). 

We first show that when $\nu \in \NN_+=\{1,2,\dots\}$, the distribution $\text{CMMB}_d(r, \nu)$ is SR.
We then prove that in any Fr\'echet class of Bernoulli vectors with equal marginal means $p$, there exists a chain with respect to the supermodular order (\cite{muller2002comparison}), denoted by $\preceq_{sm}$, of $\text{CMMB}_d(r, \nu)$ pmfs, ranging from the minimum in supermodular order to the upper Fréchet bound. We also show that $\nu$ is the parameter that drives dependence and $r$ is constrained by the Fr\'echet class.
For any given $p$ and $\nu$, we can find $r$ following an easy optimization procedure, as shown in Example \ref{ex:chain}. 
Furthermore, in the particular case $p=1/2$, it is easy to see that $r=1/2$.

In Section~\ref{sec:Preliminaries}, we introduce the Conway--Maxwell (CM) families $\text{CMB}_d(r, \nu)$ and $\text{CMMB}_d(r, \nu)$, together with the dependence notions considered in this work: the SR property, the convex order, and the supermodular order. 
For an overview of these dependence concepts and their interrelations, see \cite{borcea2009negative} and \cite{cossette2025extremal}.
The main results are presented in Section~\ref{sec:MainResults}, while some approaches to move beyond exchangeability are explored in Section~\ref{sec:FutureResearch}.

\section{Definitions and preliminaries} \label{sec:Preliminaries}
\subsection{The Conway--Maxwell-Binomial and multivariate Bernoulli distributions}
\label{sec:CMMBdistributions}

Let $\mathcal{D}$ be the class of discrete distributions on $\{0,\ldots, d\}$ and $\mathcal{D}(\mu) \subset \mathcal{D}$ the class of discrete probability mass functions (pmfs) with mean $\mu$. 
Let now $\mathcal{B}_d$ be the class of $d$-dimensional Bernoulli distributions and $\mathcal{B}_d(p)\subset \BBB_d$ the Fr\'echet class of $d$-dimensional Bernoulli pmfs with common mean $p$.
Given a random variable $W$ and a random vector $\II$, the notations $W \in \mathcal{D}(\mu)$ and $\II\in \BBB_d(p)$ means that $W$ and $\II$ have pmfs in $\mathcal{D}(\mu)$ and $\BBB_d(p)$, respectively.
For any $W \in \mathcal{D}(\mu)$, there are many Bernoulli random vectors $\II_W =(I_{W,1},\ldots I_{W,d})\in \BBB_d(\mu/d)$ with $f_W$ as the pmf of $\sum_{j=1}^d I_{W,j}$, but only one of them is exchangeable, and we denote it by $\II^e_W$.
The vector $\II_W^e$ has joint pmf $f^e_W$ in the class $\EEE_d$ of exchangeable Bernoulli pmfs, and in particular in the class $\EEE_d(\mu/d)\subset \mathcal{E}_d$ of exchangeable Bernoulli pmfs with mean $\mu/d$.

The pmf of a discrete random variable $W\in \mathcal{D}$ having the CMB distribution, denoted $W \sim \text{CMB}_d(r,\nu)$, for $r \in (0,1)$ and $\nu \in \mathbb{R}$, is given by
\begin{equation} \label{eq:pmf_CMB}
    f_W(k)=\mathbb{P}(W=k) 
    =
    \frac{\binom{d}{k}^{\nu} r^k (1-r)^{d-k}}{S_d(r,\nu)},
    \quad
    k \in \{0,1,\dots,d\},
\end{equation}
where
\begin{equation*}
    S_d(r,\nu) = \sum_{k=0}^d \binom{d}{k}^{\nu} r^k (1-r)^{d-k}.
\end{equation*}
When $\nu=1$, the binomial distribution results.
As $\nu \to +\infty$, $W$ piles up at $d/2$ if $d$ is even, or at $(d-1)/2$ and $(d+1)/2$ if $d$ is odd, while, as $\nu \to -\infty$, the mass is at $W=0$ and $W=d$.
The expected value of $W$ is given by
\begin{equation*}
    \E[W] = \sum_{k=0}^d k \PP(W=k) = \sum_{k=0}^d \frac{\binom{d}{k}^{\nu} k r^k (1-r)^{d-k}}{S_d(r,\nu)}.
\end{equation*}
Finally, let us introduce the univariate polynomials
\begin{equation} \label{eq:polynomial_G}
    \GGG_{d,\nu}(z) = \sum_{k=0}^d \binom{d}{k}^{\nu} z^k, \quad z \in \mathbb{C}.
\end{equation}
The probability generating function (pgf) of a discrete random variable $W$ is defined as $\PPP_W(z) = \E[z^W] = \sum_{k=0}^d f_W(k) z^k$, for $z \in \mathbb{C}$. 
If $W \sim \text{CMB}_d(r,\nu)$, its pgf has been derived in Section 5 of \cite{kadane2016sums}, and can be expressed in terms of the polynomials in \eqref{eq:polynomial_G} as follows,
\begin{equation*}
    \PPP_W(z) = \frac{(1-r)^{d}}{S_d(r,\nu)} \sum_{k=0}^d \binom{d}{k}^{\nu} \bigg(\frac{r}{1-r} \bigg)^k z^k
    = \frac{(1-r)^d}{S_d(r,\nu)}\GGG_{d,\nu}\left(\frac{r}{1-r}z\right)
    , \quad z \in \mathbb{C}.
\end{equation*}

Let $\II_W^e\in \EEE_d$ be the exchangeable Bernoulli vector underlying $W \sim \text{CMB}_d(r, \nu)$; we name its distribution $\text{CMMB}_d(r, \nu)$ with parameters $r, \nu$ (see also \cite{kadane2016sums}).
For the exchangeable Bernoulli random vectors, the two limit distributions for $\nu \to +\infty$ and $\nu \to -\infty$ respectively represent extreme negative and positive dependence, as we discuss below; see also \cite{frostig2001comparison}.
For a fixed $d$ and given such a parametrization of $r$ in terms of $\nu$ and $d$, $\nu$ is a true dependence parameter.
Obviously, if $\E[W] = \mu$, then $\II_W^e\in \EEE_d(\mu/d)$.
As mentioned in \cite{kadane2016sums}, when we modify either $r$ or $\nu$, the marginals are not preserved.
However, for a fixed $p$, we show that it is possible to find $r=f(p,\nu, d)$ such that $\text{CMMB}_d(r, \nu)$ is a subclass of the Fréchet class $\BBB_d(p)$, making it more attractive for modelling purposes.

Considering the case $r = 1/2$, we obtain
\begin{equation*}
    \E[W] 
    = \sum_{k=0}^d k \PP(W=k) = \sum_{k=0}^d \frac{\binom{d}{k}^{\nu} k (1/2)^{d}}{S_d(1/2,\nu)} 
    = \frac{\sum_{k=0}^d\binom{d}{k}^{\nu} k (1/2)^{d}}{\sum_{k=0}^d \binom{d}{k}^{\nu} (1/2)^{d}}
    = \frac{\sum_{k=0}^d\binom{d}{k}^{\nu} k}{\sum_{k=0}^d \binom{d}{k}^{\nu}},
\end{equation*}
and, using the symmetry of the binomial coefficient, we find
\begin{equation} \label{eq:Expectation1/2}
    2\E[W] 
    = \frac{\sum_{k=0}^d\binom{d}{k}^{\nu} k}{\sum_{k=0}^d \binom{d}{k}^{\nu}} + \frac{\sum_{k=0}^d\binom{d}{k}^{\nu} (d-k)}{\sum_{k=0}^d \binom{d}{k}^{\nu}}
    = \frac{\sum_{k=0}^d\binom{d}{k}^{\nu} d}{\sum_{k=0}^d \binom{d}{k}^{\nu}} = d.
\end{equation}
Therefore, $\E[W] = d/2$ for every $W \sim \text{CMB}_d(1/2,\nu)$, for any $\nu \in \RR$. 
The exchangeable Bernoulli vector $\II_W^e$ behind $W\sim \text{CMB}_d(1/2, \nu)$ has Bernoulli marginals with mean $1/2$, that is $\II^e_W \in \EEE_d(1/2)$; thus, the $\text{CMMB}_d(1/2, \nu)$ distributions belong to the same class $\EEE_d(1/2)$. 
This implies that we can compare the strength of association between two $\text{CMMB}_d(1/2, \nu)$ distributions using a dependence order, as these are partial orders defined on Fr\'echet classes of distributions with the same one-dimensional margins; see \cite{muller2001stochastic}.

\subsection{The strongly Rayleigh property and dependence orders}
\label{sec:SRpropertyAndDependence}

In \cite{borcea2009negative}, the authors introduce the SR property, a strong notion of negative dependence for Bernoulli vectors, defined in terms of their pgfs.
Let $\II \in \BBB_d$, its joint pgf is defined as
\begin{equation*}
    \mathcal{P}_{\II}(\zz) 
    =
    \E[\zz^{\II}]
    =
    \sum_{\ii \in \{0,1\}^d} f_{\II}(\ii)z_1^{i_1} \cdots z_d^{i_d}, 
    \quad \zz = (z_1,\dots,z_d) \in \mathbb{C}^d,
\end{equation*}
where $\zz^{\II} := \prod_{j=1}^d z_j^{I_j}$.
The joint pgf of a Bernoulli random vector is a polynomial in which each variable has degree at most one; such polynomials are called multi-affine.
We first recall a preliminary notion about polynomials, on which the SR property is based.
\begin{definition}\label{def:stab}
    A non-zero polynomial $\mathcal{P} \in \mathbb{C}[z_1,\dots,z_d]$ is stable if $\mathcal{P}(z_1,\dots,z_d) \neq 0$ whenever $\zz = (z_1,\dots,z_d) \in \HHH^d$, where $\HHH = \{ z \in \mathbb{C} : \Im(z)>0 \}$, and where $\Im(z)$ denotes the imaginary part of the complex number $z$. If the coefficients of $\mathcal{P}$ are real, then $\mathcal{P}$ is real stable.
\end{definition}
We are interested in polynomials that are pgfs of $d$-dimensional Bernoulli  random vectors and discrete random variables, i.e., of multi-affine or univariate polynomials with non-negative and real coefficients.  
Theorem 5.6 of \cite{branden2007polynomials} states that a multi-affine polynomial $\PPP(\zz)$ with real coefficients is real stable if and only if, for every $1 \leq j_1 < j_2 \leq d$,
\begin{equation} \label{eq:CondStrongRayleigh}
    \frac{\partial \PPP}{\partial z_{j_1}} (\xx) \frac{\partial \PPP}{\partial z_{j_2}} (\xx) \geq \frac{\partial^2 \PPP}{\partial z_{j_1} \partial z_{j_2}} (\xx) \PPP(\xx), \quad \text{for all } \xx \in \RR^d.
\end{equation}
See \cite{wagner2011multivariate} for more details on stable and multi-affine polynomials.
In the case of univariate polynomials with real coefficients, real stability is equivalent to having all real zeros, a property we will refer to as hyperbolicity throughout this paper.

We are now ready to define the SR property.
\begin{definition}
    A random vector $\II \in \mathcal{B}_d$ or a random variable $W\in \mathcal{D}_d$ satisfies the SR property if its pgf is real stable.
\end{definition}
If a random vector or a random variable satisfies the SR property, we say that its pmf also satisfies the SR property. Note that if W is SR, then its distribution is a Poisson Binomial for some $p_1,\dots,p_d$, see Corollary 4.2 in \cite{tang2022poisson}.

The SR property implies the most famous notion of negative dependence, called negative association. 
Negative association has a natural probabilistic definition.
\begin{definition} \label{def:NA}
    A $d$-dimensional random vector $\II\in \mathcal{B}_d(\pp)$ is said to be negatively associated (NA) if for any two disjoint sets $\Lambda_1, \Lambda_2 \subseteq \{1,\dots,d\}$ and two monotone increasing functions $h_1, h_2$, the following inequality holds
    \begin{equation*}
        \E[h_1(I_j: j \in \Lambda_1)h_2(I_j: j \in \Lambda_2)] \leq \E[h_1(I_j: j \in \Lambda_1)] \E[h_2(I_j: j \in \Lambda_2)],
    \end{equation*}
    provided that the expectations are finite.
\end{definition}
Our investigation focuses on the SR property as it is not only stronger, but often also easier to verify than NA (\cite{borcea2009negative}).
Theorem~3.8 in \cite{borcea2009negative} states that $W$ is SR if and only if the exchangeable Bernoulli vector $\II^e_W$ behind $W$ is SR, and the condition that $W$ is SR  is sometimes more convenient to verify. 

The strength of dependence is measured using partial orders within the Fr\'echet class $\BBB_d(p)$; indeed, only random vectors with the same margins can be compared; see, e.g., Chapter 3 in \cite{muller2002comparison}. 
Here we consider the supermodular order $\preceq_{sm}$ for Bernoulli vectors $\II\in \BBB_d(p)$ that is linked to the convex order for discrete random variables $W\in \mathcal{D}(dp)$, and we need to introduce both.
A supermodular function is a function $\phi \colon \{0,1\}^d \to \RR$ such that $\phi(\ii)+\phi(\kk) \leq \phi(\ii \wedge \kk)+\phi(\ii \vee \kk)$, where $\wedge$ and $\vee$ respectively denotes the componentwise minimum and maximum operators. 
\begin{definition}
    Given two random vectors $\II=(I_1,\dots,I_d)$ and $\boldsymbol{K}=(K_1,\dots,K_d)$ with joint pmfs in the same Fréchet class, $\II$ is said to be smaller than $\KK$ in the supermodular order, denoted by \(\boldsymbol{I} \preceq_{sm} \boldsymbol{K}\), if $\E[\phi(\boldsymbol{I})] \le \E[\phi(\boldsymbol{K})]$, for all supermodular functions $\phi \colon \{0,1\}^d \to \RR$ such that the expectations are finite.
\end{definition}

\begin{definition}
    Given two random variables $X$ and $Y$, we say that $X$ is smaller than $Y$ under the convex order, denoted by $X \preceq_{cx} Y$, if $\E[\varphi(X)] \leq \E[\varphi(Y)]$ for every convex function $\varphi:\RR \to \RR$, for which the expectations are finite.
\end{definition}

Since, as proved in Section 3 of \cite{frostig2001comparison}, $W\preceq_{cx}W'$ implies $I^e_W\preceq_{sm} I^e_{W'}$,  we can first order discrete random variables and then the exchangeable vectors behind them.

\section{The CM families and the SR property}
\label{sec:MainResults}

The authors of \cite{kadane2013number} state that random vectors following a $\text{CMMB}_d(r,\nu)$ distribution exhibit positive pairwise correlation if $\nu < 1$ and negative pairwise correlation if $\nu > 1$, with the special case $\nu=1$ recovering independence. 
As already observed in Section \ref{sec:CMMBdistributions}, the parameter $\nu$ drives the dependence that is, unusually, decreasing in $\nu$.

We want to study whether the CMMB distribution satisfies the SR property for $\nu \geq 1$. The first step is to study if the CMB distribution does.
We first notice that if $\nu = 1$, then $W$ satisfies the SR property as $\PPP_W(z)$ has only real zeros, as it always happens when we consider independent rvs (see \cite{borcea2009negative}).

The next counterexample shows that the CMB distribution is not always SR for $\nu\geq 1$. 

\begin{example} \label{ex:NotSR}
    We consider the case $d=4$, $r=1/2$, and $\nu \geq 1$. 
    The pgf of the random variable $W \sim \text{CMB}_4(1/2,\nu)$ is given by the univariate polynomial $\PPP_W(z) = C (1 + 4^{\nu} z + 6^{\nu} z^2 + 4^{\nu} z^3 + z^4)$, where $C = 1/[2^4 \cdot S_4(1/2, \nu)] \in \RR$. 
    As $\PPP_W(z)$ has strictly positive coefficients, it follows that its value is strictly positive for every non-negative real argument, i.e., no non-negative real number can be a zero of the polynomial.
    Considering $z < 0$, we can write
    \begin{equation*}
        \frac{\PPP_W(z)}{z^2} 
        = 
        C \bigg[ \bigg(z^2 + \frac{1}{z^2} \bigg) + 4^{\nu} \bigg(z + \frac{1}{z}\bigg) + 6^{\nu} \bigg].
    \end{equation*}
    Using the change of variable $y = z+1/z$, we define the function $Q(y,\nu) = y^2 + 4^{\nu}y + (6^{\nu} - 2)$, and we notice that if $y_1 \in \RR$ is such that $Q(y_1,\nu) = 0$, then we find two zeros of $\PPP_W(z)$ by
    \begin{equation*}
        z_{1,j} = \frac{y_1 + (-1)^j\sqrt{y_1^2-4}}{2},
    \end{equation*}
    for $j \in \{1,2\}$.
    If $|y_1| < 2$, it follows that $z_{1,1}$ and $z_{1,2}$ are complex roots.
    Moreover, by Bolzano's theorem, if $Q(2,\nu)Q(-2,\nu) < 0$, then there exists $y_1 \in (-2,2)$ such that $Q(y_1, \nu) = 0$ and, therefore, $\PPP_W(\zz)$ admits complex roots.
    Since $Q(2,\nu) = 2 + 2 \cdot 4^{\nu} + 6^{\nu} > 0$ for every $\nu \in \RR$, we need to study the sign of the continuous function $H(\nu) := Q(-2,\nu) = 2 - 2 \cdot 4^{\nu} + 6^{\nu}$, as $\nu$ varies in $\RR$.
    Since $H(1) = 0$, $H(2) > 0$, and $H'(1)=\log(6^6/2^{16}) < 0$, where $H'(\nu) = \mathrm{d}H(\nu)/\mathrm{d}\nu$, it follows that there exists $\bar{\nu} \in (1,2)$ such that $H(\bar\nu)=0$ and $H(\nu)<0$ for every $\nu \in (1,\bar{\nu})$.
    Therefore, for every $\nu \in (1,\bar{\nu})$, $\PPP_W(z)$ admits complex roots and $W \sim \text{CMB}_4(1/2,\nu)$ does not satisfy the SR property.
    Numerically, one finds $\bar{\nu} \cong {1.14772}$.
\end{example}
To prove our main result on the SR property, that is to give the conditions on the parameters $r$ and $\nu$ for the distribution $\text{CMB}_d(r,\nu)$ to be SR, we need some preliminary propositions.
\begin{proposition} \label{prop:parameters_SRproperty}
    If $W \sim \text{CMB}_d(r,\nu)$ satisfies the SR property, then $W' \sim \text{CMB}_d(r',\nu)$ satisfies the SR, for every choice of $r' \in (0,1)$. 
\end{proposition}
\begin{proof}
    The pgfs of $W$ and $W'$ are respectively given by
    \begin{equation*}
        \PPP_W(z)
        = \frac{(1-r)^d}{S_d(r,\nu)}\GGG_{d,\nu}\left(\frac{r}{1-r}z\right)
    \end{equation*}
    and
    \begin{equation*}
        \PPP_{W'}(z) 
        =
        \frac{(1-r')^d}{S_d(r',\nu)}\GGG_{d,\nu}\left(\frac{r'}{1-r'}z\right)
        =
        \bigg( \frac{1-r'}{1-r} \bigg)^d \frac{S_d(r,\nu)}{S_d(r',\nu)} \PPP_W \bigg( \frac{1-r}{r}\frac{r'}{1-r'} z \bigg).
    \end{equation*}
    Since $W$ satisfies the SR property, its pgf $\PPP_W(z)$ is hyperbolic, it is easy to show that $\PPP_{W'}(z)$ is hyperbolic and thus that $W'$ satisfies the SR property; see the stability preservers in Lemma 2.4 in \cite{wagner2011multivariate}.
\end{proof}

A consequence of Proposition \ref{prop:parameters_SRproperty} is that the SR property is independent of the parameter $r$.
Hence, the conclusions in Example \ref{ex:NotSR} can be extended to any $r \in (0,1)$.
Moreover, we can prove our results on the SR property for $r=1/2$ without loss of generality.
This choice simplifies the proofs, thanks to the following corollary. 
\begin{corollary} \label{cor:PolynomialSR}
    The random variable $W \sim \text{CMB}_d(r,\nu)$ with $r \in (0,1)$ and $\nu \in \mathbb{R}$ satisfies the SR property if and only if the polynomial $\GGG_{d,\nu}(z)$ in \eqref{eq:polynomial_G} is hyperbolic.
\end{corollary}
\begin{proof}
    By Proposition~\ref{prop:parameters_SRproperty} the SR property is independent of $r$, therefore the random variable $W$ satisfies the SR property if and only if $W' \sim \text{CMB}_d(1/2,\nu)$ satisfies the SR property.
    The pgf of $W'$ is given by
    \begin{equation*}
        \PPP_{W'}(z) = \frac{(1/2)^{d}}{S_d(1/2,\nu)} \GGG_{d,\nu}(z),
    \end{equation*}
    that is hyperbolic if and only if the polynomial $\GGG_{d,\nu}(z)$ in \eqref{eq:polynomial_G} is hyperbolic.
\end{proof}
The proof of our main result is based on the following lemma.
\begin{lemma} \label{prop:nu_plus1}
    If $W \sim \text{CMB}_d(r,\nu)$ with $r \in (0,1)$ and $\nu \in \mathbb{R}$ satisfies the SR property, then $W' \sim \text{CMB}_d(r,\nu+1)$ satisfies the SR property.
\end{lemma}
\begin{proof}
    Using Corollary~\ref{cor:PolynomialSR}, the statement reduces to show that if $\GGG_{d,\nu}(z)$ is hyperbolic, then $\GGG_{d,\nu+1}(z)$ is hyperbolic.
    Observe that $\GGG_{d,\nu+1}(z) = T[\GGG_{d,\nu}(z)]$, where T is the linear operator on $\mathbb{C}[z]$ defined by $T[z^k] = \lambda(k) z^k$, with $\lambda \colon \mathbb{N} \to \mathbb{R}$ given by
    \begin{equation*}
        \lambda(k) =
        \begin{cases}
            \binom{d}{k}, &\text{if } 0 \leq k \leq d,
            \\
            0, &\text{if } k > d.
        \end{cases}
    \end{equation*}

   It is therefore sufficient to prove that the operator $T$ preserves hyperbolicity.
    In \cite{schur1914zwei} (see also Theorem 1.7 in \cite{borcea2010multivariate}), P\'olya and Schur show that $T$ is a hyperbolicity preserver if and only if $T[(1+z)^n]$ is hyperbolic with all zeros of the same sign, for all nonnegative integers $n$.
    We have
    \begin{equation*}
        T[(1+z)^n] = T \bigg[ \sum_{k=0}^n \binom{n}{k} z^k \bigg] = \sum_{k=0}^n \binom{n}{k} T[z^k] = \sum_{k=0}^{\min\{n,d\}} \binom{n}{k} \binom{d}{k} z^k.
    \end{equation*}
    Consider $n \in \{0,\dots,d-1\}$.
    Theorem 3.2$(v)$ in \cite{driver2008zeros} states that the hypergeometric polynomial, defined by
    \begin{equation*}
        F(-n,b;c;z) = \sum_{k=0}^{\infty} \frac{(-n)_k (b)_k}{(c)_k} \frac{x^k}{k!},
    \end{equation*}
    with $n \in \mathbb{N}$, $b,c \in \mathbb{R}$, $c>0$, and $b<-n$, is hyperbolic with all negative zeros; we write $(a)_k$ to denote the Pochhammer's symbol, i.e., $(a)_k = a(a+1) \cdots (a+k-1)$ for $k \in \mathbb{N}$, and $(a)_0 = 1$, for every $a \in \mathbb{C}$.
    By setting $b = -d$ and $c = 1$, we have
    \begin{equation*}
    \begin{split}
        F(-n,-d;1;z) 
        &= 
        \sum_{k=0}^{\infty} \frac{(-n)_k (-d)_k}{(1)_k} \frac{x^k}{k!}
        \\
        &= 
        \sum_{k=0}^{n} \frac{(-n)(-n+1) \cdots (-n+k-1) (-d)(-d+1) \cdots (-d+k-1)}{k!} \frac{x^k}{k!}
        \\
        &= 
        \sum_{k=0}^n (-1)^{2k} \frac{n!}{k!(n-k)!} \frac{d!}{k!(d-k)!} x^k = \sum_{k=0}^{n} \binom{n}{k} \binom{d}{k} x^k = T[(1+z)^n].
    \end{split}
    \end{equation*}
    Therefore, for $n \in \{0,\dots,d-1\}$, the polynomial $T[(1+z)^n]$ is hyperbolic with all negative zeros.
    When $n \in \{d+1,d+2,\dots\}$, by applying again Theorem 3.2$(v)$ of \cite{driver2008zeros}, we have that $T[(1+z)^n] = F(-d,-n;1;z)$ is hyperbolic with all negative zeros.
    It remains to show that $T[(1+n)^d]$ is hyperbolic with all negative zeros.
    For every $z \neq 1$, we can write
    \begin{equation} \label{eq:Relation_Legendre}
        T[(1+z)^d] = \sum_{k=0}^d \binom{d}{k}^{2} z^k = (1-z)^d \mathcal{L}_d \bigg( \frac{1+z}{1-z} \bigg),
    \end{equation}
    where $\mathcal{L}_d(x)$ is the Legendre polynomial of order $d$.
    Indeed, from \cite[p.162]{rainville1960special}, the Legendre polynomial $\mathcal{L}_d(x)$ can be written as
    \begin{equation*}
        \mathcal{L}_d(x) = \sum_{k=0}^d \binom{d}{k}^2 \bigg( \frac{x-1}{2} \bigg)^{d-k} \bigg( \frac{x-1}{2} \bigg)^k,
    \end{equation*}
        and, after standard computations, \eqref{eq:Relation_Legendre} follows.
    It is well-known that $\mathcal{L}_d(x)$ has $d$ real and distinct zeros, all of which lie in the interval $(-1,1)$; see, e.g., \cite{turan1950zeros}.
    We denote the zeros of $\mathcal{L}_d(x)$ by $x_1, \dots, x_d$.
    Therefore, the zeros of $T[(1+z)^d]$ are given by
    \begin{equation*}
        z_j = \frac{x_j-1}{x_j+1},
    \end{equation*}
    for $j \in \{1,\dots,d\}$, and we can conclude that $z_1, \dots, z_d$ are real, distinct, and lies in $(-\infty, 0)$.
    Thus, $T[(1+z)^d]$ is hyperbolic with all negative zeros, and this completes the proof.
\end{proof}

We can now state our main result about the SR property.
\begin{theorem}
    The random variable $W \sim \text{CMB}_d(r,\nu)$ with $r \in (0,1)$ and $\nu \in \NN_+$ satisfies the SR property. The random vector $\II^e_W\sim \text{CMMB}_d(r, \nu)$ behind satisfies the SR property.
\end{theorem}
\begin{proof}
    Using Proposition \ref{prop:nu_plus1}, the claim follows by induction since the random variable $W' \sim \text{CMB}_d(r,1)$, with $r \in (0,1)$, satisfies the SR property. 
    The exchangeable vector behind any SR random variable satisfies the SR property by Theorem 3.8 in \cite{borcea2009negative}.
\end{proof}

In the following proposition, we  provide conditions to construct a chain with respect to the convex order in $\mathcal{D}(\mu)$, that is the first step to build a chain with respect to the supermodular order in $\BBB_d(\mu/d)$.  

\begin{proposition} \label{prop:CXorder}
Let $W_1 \sim \text{CMB}_d(r_1,\nu_1)$ and $W_2 \sim \text{CMB}_d(r_2,\nu_2)$ such that $\E[W_1] = \E[W_2]=\mu$ and $\nu_1 \leq \nu_2$.
    Then, $W_2 \preceq_{cx} W_1$ holds.
\end{proposition}
\begin{proof}
We denote by $f_1$ and $f_2$ the pmfs and by $F_1$ and $F_2$ the cdfs of $W_1$ and $W_2$, respectively.
Let also denote $g:=f_1-f_2$ and $G := F_1-F_2$.

Since $\E[W_1] = \E[W_2]$, the cdfs $F_1$ and $F_2$ cross at least once (see the proof of Theorem 3.A.5 in \cite{shaked2007}), and $G$ changes sign at least once. 
Since $G(n) = \sum_{k=0}^n g(k)$ and $G(d)=0$, if $G$ changes sign at least once, then $g$ changes sign at least twice. If we prove that $g$ changes sign exactly twice, with sign sequence $+,-,+$, by Theorem 3.A.44 in \cite{shaked2007}, we have $W_2\preceq_{cx}W_1$. The sign changes of $g$ coincide with those of $\ell:= \ln(f_1/f_2)$. Indeed,
\begin{equation*}
    \ell(k)>0 
    \iff
    \ln\frac{f_1(k)}{f_2(k)}>0 
    \iff 
    \frac{f_1(k)}{f_2(k)}>1 
    \iff
    f_1(k)>f_2(k) 
    \iff
    g(k)>0.
\end{equation*}
Moreover, $\ell(k)$ is a discrete convex function, i.e., $\ell(k+1) + \ell(k-1) - 2\ell(k) \geq 0$.
Indeed, from \eqref{eq:pmf_CMB}, for any $k \in \{0,1,\dots,d\}$, we have
    \begin{equation*}
    \frac{f_1(k)}{f_2(k)} 
    =
    \left(\frac{1-r_1}{1-r_2}\right)^d\frac{S_d(r_2,\nu_2)}{S_d(r_1,\nu_1)}\binom d k^{\nu_1-\nu_2}\left(\frac{r_1(1-r_2)}{r_2(1-r_1)}\right)^k,
    \end{equation*}
thus
\begin{equation*}
    \ell(k)
    =d\ln
    \left(\frac{1-r_1}{1-r_2}\right)+\ln\left(\frac{S_d(r_2,\nu_2)}{S_d(r_1,\nu_1)}\right)+{(\nu_1-\nu_2)}\ln\binom d k+k\ln\left(\frac{r_1(1-r_2)}{r_2(1-r_1)}\right),
    \end{equation*}
and
\begin{align*}
    \ell(k+1)+\ell(k-1)-2\ell(k)&=(\nu_1-\nu_2) \mathrm{ln}\left( \binom{d}{k+1} \binom{d}{k-1} \binom{d}{k}^{-2} \right) \\
    &=(\nu_1-\nu_2)\mathrm{ln}\left( \frac{k}{k+1}\frac{d-k}{d-k+1}\right) 
    \geq 0
\end{align*}
for any $k$.
Therefore, since $\ell$ is a convex function, it changes sign at most twice with sign sequence $+,-,+$, and so does $g$. We can conclude that $g$ changes sign exactly twice, with sign sequence $+,-,+$, and $W_2\preceq_{cx}W_1$.
\end{proof}

\begin{corollary}\label{cor:sm}
    Let $\II_1^e$ and $\II_2^e$ be two exchangeable Bernoulli random vectors such that
    \begin{equation*}
        \sum_{j=1}^d I_{1,j}^e  \overset{\mathcal{L}}{=} W_1 
        \quad \text{and} \quad
        \sum_{j=1}^d I_{2,j}^e  \overset{\mathcal{L}}{=} W_2,
    \end{equation*}
    where $W_1 \sim \text{CMB}_d(r_1,\nu_1)$ and $W_2 \sim \text{CMB}_d(r_2,\nu_2)$, with $\E[W_1] = \E[W_2] = \mu$ and $\nu_1 \leq \nu_2$, and $\overset{\mathcal{L}}{=}$ means equality in distribution.
    Then, $\II_1^e, \II_2^e\in \BBB_d(p)$, with $p=\mu/d$, and $\II_2^e \preceq_{sm} \II_1^e$.
\end{corollary}

\begin{proof}
    By Proposition \ref{prop:CXorder}, we have $W_2 \preceq_{cx} W_1$, which implies $\II_2^e \preceq_{sm} \II_1^e$; see Section~3 of \cite{frostig2001comparison}.
\end{proof}
The above corollary proves that, for any given $\mu \in (0,d)$, there is a chain with respect to $\preceq_{sm}$, included in the Fréchet class $\BBB_d(p)$, with $p=\mu/d$, from the minimum under the supermodular order to the upper Fr\'echet bound.

\begin{example}\label{ex:chain} 
For a fixed triple $(d,\nu,p)$, the value $r$ is obtained such that $\E[W] = dp$ by optimization. For example, assume that $p = 1/3$ and $d = 9$. 
For $\nu \in \{2,3,4,5\}$, the corresponding values of $r$ and $\text{Var}(W)$ are $\{0.21367747, 0.12810820,  0.07352747, 0.04105203\}$ and  $\{1.0748125, $$0.7319828,$$ 0.5547126,$\newline$0.4452527\}$, respectively. 
The values of $\text{Var}(W)$ are decreasing in $\nu$ as a consequence of Corollary~\ref{cor:sm}. 
When $\nu = 1$ (independence), $r = 1/3$ and $\text{Var}(W) = 2$. 
Figure~\ref{fig:ExamplePmfW} shows the (approximated) values of the pmfs, which, as expected, are more concentrated around the mean as $\nu$ increases.
\end{example}   
\begin{figure}[tb]
    \centering
    \includegraphics[width=0.45\linewidth]{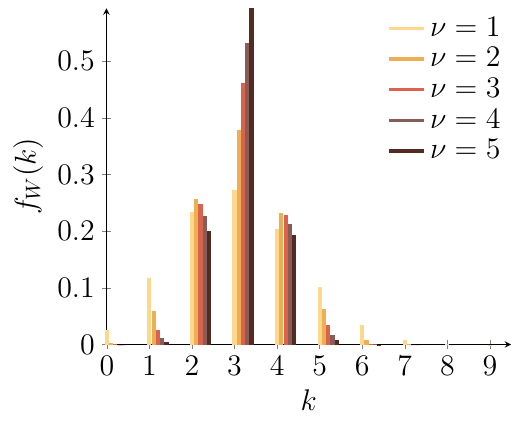}
    \caption{Histogram of the values of pmfs from Example~\ref{ex:chain}. }
    \label{fig:ExamplePmfW}
\end{figure}

Since we do not have an analytical expression for $r$ as a function of $p$, $r$, and $\nu$, we cannot explicitly find the chain of pmfs unless we choose $r=1/2$.
For this reason, we now consider the CM families with $r=1/2$. Indeed, as discussed at the end of Section~\ref{sec:CMMBdistributions}, if $r=1/2$ the exchangeable Bernoulli vector $\II^e_W$ behind $W\sim \text{CMB}_d(1/2, \nu)$ belongs to $\EEE_d(1/2)$.
The above corollary proves that $\text{CMMB}(1/2, \nu)$ is a chain with respect to the supermodular order in $\mathcal{B}_d(1/2)$. 
As $\nu\rightarrow +\infty$, the $\text{CMB}$ distribution converges to a distribution with support on $\{(d-1)/2, (d+1)/2\}$ if $d$ is odd, and on $\{d/2\}$ if $d$ is even.
Conversely, as $\nu\rightarrow -\infty$, it converges to a distribution with support on $\{0,d\}$. 
Consequently, the $\text{CMMB}(1/2, \nu)$  converges to the minimum under the supermodular order as $\nu\rightarrow +\infty$ and to the upper Fr\'echet bound as $\nu\rightarrow -\infty$ (see \cite{frostig2001comparison}). 
Therefore, $\text{CMMB}_d(1/2, \nu)$ is a well-ordered chain with respect to the supermodular order in $\mathcal{B}_d(1/2)$. 
This means that the family is able to model each level of dependence with only one parameter.
Furthermore, the subfamily $\text{CMMB}_d(1/2, n)$, for $n \in \NN_+$, is a chain under the supermodular order of pmfs that satisfy the SR property.

\section{Extension and further research}
\label{sec:FutureResearch}

This section presents two ways to move from the $\text{CMMB}_d(1/2, \nu)$ distribution to obtain a $d$-dimensional SR and non-exchangeable Bernoulli vector. 
The first approach relies on a geometrical representation of the Bernoulli pmfs and allows us to include in the family the limit cases $\nu\to \pm \infty$. 
The second approach is based on a stochastic representation and allows us to choose the new vector belonging to a given  Fr\'echet class, i.e., with given marginal means.
\subsection{Geometrical approach}
Consider a random variable $W \sim \text{CMB}_d(1/2,\nu)$, with $\nu \in \RR$. Since its expected value is $d/2$, we have $W \in \mathcal{D}(d/2)$. The authors of \cite{fontana2021model} show that $\mathcal{D}(\mu)$, for $\mu \in (0,d)$, is a convex polytope with extremal points
\begin{equation*}
    r^D_{j_1,j_2}(k) =
    \begin{cases}
        \frac{j_2 - \mu}{j_2 - j_1}, &\text{for } k=j_1,
        \\
        \frac{\mu - j_1}{j_2 - j_1}, &\text{for } k=j_2,
        \\
        0, &\text{otherwise},
    \end{cases}
\end{equation*}
for any $j_1 \in \{0,\dots,\lfloor \mu \rfloor\}$ and $j_2 \in \{\lceil \mu \rceil, \dots, d \}$, where $\lfloor \mu \rfloor$ is the highest integer lower than or equal to $\mu$ and $\lceil \mu \rceil$ is the lowest integer higher than $\mu$.
Therefore, we have 
\begin{equation*}
    f_W = \sum_{j_1=0}^{\lfloor d/2 \rfloor} \sum_{j_2 = \lceil d/2 \rceil}^d \lambda_{j_1,j_2} r^D_{j_1,j_2},
\end{equation*}
where $\lambda_{j_1,j_2}$, for $j_1 \in \{0,\dots,\lfloor \mu \rfloor\}$ and $j_2 \in \{\lceil \mu \rceil, \dots, d \}$, are non-negative real numbers summing to 1.
Moreover, we also know that $f_W$ is symmetric, i.e., $f_W(k) = f_W(d-k)$, for every $k \in \{0,\dots,d\}$.
This implies that we can write $f_W$ as a convex combination involving only the symmetric extremal points,
\begin{equation} \label{eq:Repr_CMB12_Symm}
    f_W(k)
    =
    \sum_{j=0}^{\lfloor d/2 \rfloor} \lambda_{j} r^D_{j}(k),
\end{equation}
where $\lambda_j \geq 0$, $\lambda_0 + \dots + \lambda_{\lfloor d/2 \rfloor} = 1$, and, for $j \in \{0, \dots, \lfloor d/2 \rfloor \}$, 
\begin{equation*}
    r^D_{j}(k) =
    \begin{cases}
        \frac{1}{2}, &\text{for } k \in \{j,d-j\},
        \\
        0, &\text{otherwise}.
    \end{cases}
\end{equation*}
By equating \eqref{eq:Repr_CMB12_Symm} and \eqref{eq:pmf_CMB} with $r=1/2$, we get, for $j \in \{0, \dots, \lfloor d/2 \rfloor \}$,
\begin{equation} \label{eq:weightsCMB}
    \lambda_j = \frac{2\binom{d}{j}^{\nu}}{\sum_{y=0}^d \binom{d}{y}^{\nu}}.
\end{equation}
Since there exists a one-to-one correspondence between exchangeable Bernoulli distributions and discrete distributions, the exchangeable Bernoulli distribution $f^e$ of a random vector $\II^e$ such that $I^e_1 + \dots + I^e_d \overset{\mathcal{L}}{=} W$ can be written as
\begin{equation}\label{eq:fe}
    f^e
    =
    \sum_{j=0}^{\lfloor d/2 \rfloor} \lambda_{j} r^e_j,
\end{equation}
where 
\begin{equation*}
    r^e_{j}(\ii) =
    \begin{cases}
        \frac{1}{2}\binom{d}{j}^{-1}, &\text{if } i_1+\dots+i_d=j,
        \\
        \frac{1}{2}\binom{d}{j}^{-1}, &\text{if } i_1+\dots+i_d=d-j,
        \\
        0, &\text{otherwise}.
    \end{cases}
\end{equation*}

The above representation includes the exchangeable pmf that is minimal under the supermodular order $r^e_{\lfloor d/2 \rfloor}$ when $\lambda_{\lfloor d/2 \rfloor}=1$, and the upper Fr\'echet bound $r^e_0$ when $\lambda_0 = 1$.
The extremal point $r^e_{\lfloor d/2 \rfloor}$ lies in a polytope contained in the Fréchet class $\BBB_d(1/2)$; specifically, this is the convex polytope of all Bernoulli pmfs in $\BBB_d(1/2)$ associated with random vectors whose sums have distribution $r_{\lfloor d/2 \rfloor}^D$; see \cite{cossette2025extremal}.
As we illustrate at the end of the following example, one can select pmfs in this sub-polytope distinct from $r^e_{\lfloor d/2 \rfloor}$ to construct SR distributions that are no longer exchangeable.
\begin{example}{\bf Three-dimensional case.}
For $d = 3$, there are only two symmetric extremal points of $\mathcal{D}(3/2)$, which are given by 
\begin{equation*}
    r^D_0(k) =
    \begin{cases}
        \frac{1}{2}, &\text{if } k \in \{0,3\},
        \\
        0, &\text{if } k \in \{1,2\},
    \end{cases}
    \quad \text{and} \quad
    r^D_1(k) =
    \begin{cases}
        0, &\text{if } k \in \{0,3\},
        \\
        \frac{1}{2}, &\text{if } k \in \{1,2\},
    \end{cases}
\end{equation*}
and for $W \sim \text{CMB}_3(1/2,\nu)$, from \eqref{eq:Repr_CMB12_Symm} we have
\begin{equation}\label{eq:CMBlinelambda}
   f_{\lambda,W}(k):= f_W(k) = \lambda r^D_0(k) + (1-\lambda) r^D_1(k), \,\,\, \lambda\in[0,1],
\end{equation}
where $\lambda = 1/(1+3^{\nu})$ derives from \eqref{eq:weightsCMB}. The extremal points corresponding to  $\lambda=0$ and $\lambda=1$ are the limit cases for the CMMB distribution $f^e_{\lambda, w}$  underlying $W$, that can be included in the family using this representation. With $\lambda=0$, we have the minimal pmf under the supermodular order, if $\lambda=1$ we have the upper Fr\'echet bound and if $\lambda =0.25$ we have independence.  The extremal point $r_1^D$ is the minimal convex  pmf in $\mathcal{D}(3/2)$.  Therefore,  $f^e_{\lambda,W}$ is a chain with respect to the supermodular order. 

In what follows, we prove that the SR property holds for any $\nu\geq1$, when $d=3$. The pgf of $W$ is given by
\begin{equation*}
    \PPP_W(z) = \frac{1}{2(1+3^{\nu})}(1+z^3) + \frac{3^{\nu}}{2(1+3^{\nu})} (z+z^2).
\end{equation*}
This polynomial is hyperbolic if and only if the joint pgf $\PPP^e$ of the corresponding exchangeable Bernoulli distribution is stable, where
\begin{equation*}
    \PPP^e(z) = \frac{1}{2(1+3^{\nu})}(1+z_1z_2z_3) + \frac{3^{\nu}}{6(1+3^{\nu})} (z_1+z_2+z_3+z_1z_2+z_1z_3+z_2z_3).
\end{equation*}
Since this polynomial is multi-affine and exchangeable, it is stable if and only if condition~\eqref{eq:CondStrongRayleigh} holds for $j_1=1$ and $j_2=2$.
This condition simplifies to 
\begin{equation*}
    (3^{\nu-1} + 3^{\nu-1} x_3)^2 - (3^{\nu-1} + x_3)(1 + 3^{\nu-1} x_3) \geq 0,
\end{equation*}
for every $x_3 \in \RR$.
After standard computations, the condition is verified if and only if $\nu \geq 1$.
By Proposition \ref{prop:parameters_SRproperty}, any CMB distribution with $d=3$, $r \in (0,1)$, and $\nu \geq 1$ satisfies the SR property.

We conclude this example showing how the geometrical representation can be used to construct non-exchangeable SR pmfs.
In \cite{fontana2025geometrical}, the authors prove that the Bernoulli pmfs of vectors $\II$ whose sums $\sum_{j=1}^dI_j$ have distribution minimal in convex order are points in a convex polytope. 
They found its extremal points $r_1^{cx}$, $r_2^{cx}$, and $r_3^{cx}$, which are pmfs in $\BBB_3(1/2)$ such that $r_k^{cx}(\ee_k) = r_k^{cx}(\boldsymbol{1}-\ee_k) = 1/2$, for $k \in \{1,2,3\}$, where $(\ee_1,\ee_2,\ee_3)$ is the canonical basis of $\RR^3$ and $\boldsymbol{1} = (1,1,1)$.
From \eqref{eq:fe} we have that a CMMB pmf in $\BBB_3(1/3)$ has the representation
\begin{equation*}
    f^e_{\lambda,W}(k):= \lambda r^e_0(k) + (1-\lambda) r^e_1(k), \quad \lambda\in[0,1],
\end{equation*}
where $r^e_1$ corresponds to the minimal convex extremal point $r_1^D$.
We therefore have
\begin{equation*} 
    f^e_{\lambda, W}(\xx) = \lambda r^e_0(\xx) + (1-\lambda) (r^{cx}_1(\xx)/3 + r^{cx}_2(\xx)/3 + r^{cx}_3(\xx)/3), \quad \lambda\in[0,1].
\end{equation*}
Starting from $f^e_{\lambda,W}$, we can find a new pmf $f^*$ in a neighborhood of $f^e_{\lambda,W}\in \BBB_3(1/2)$ that is no more exchangeable, but still SR.
For example,
\begin{equation*}
    f^*(k)= 0.1r^e_0(k) + 0.9 ((1/3-0.05)r^{cx}_1(k)+1/3r^{cx}_2(k) + (1/3+0.05) r^{cx}_3(k))
\end{equation*}
belongs to $\BBB_3(1/2)$ and is SR.
This can be easily verified since in the case $d=3$, for any $p \in (0,1)$, the conditions in \eqref{eq:CondStrongRayleigh} simplify. 
After standard computations, a pmf $f \in \mathcal{B}_3(p)$ is SR if and only if, for all permutations $\pi$ of $\{0,1,2\}$, we have
\begin{equation*}
\begin{split}
    \big[ 
    &f(\ee_{\pi(0)})f(\ee_{\pi(1)} + \ee_{\pi(2)}) + f(\ee_{\pi(1)})f(\ee_{\pi(0)} + \ee_{\pi(2)}) - f(\boldsymbol{0})f(\boldsymbol{1}) - f(\ee_{\pi(2)})f(\ee_{\pi(0)} + \ee_{\pi(1
    )})  
    \big]^2
    \\
    &- 4 f(\ee_{\pi(0)})f(\ee_{\pi(1)}) \big[ f(\ee_{\pi(0)} + \ee_{\pi(2)})f(\ee_{\pi(1)} + \ee_{\pi(2)}) - f(\boldsymbol{1})f(\ee_{\pi(2)}) \big] \geq 0,
\end{split}
\end{equation*}
where $\boldsymbol{0} = (0,0,0)$.

\end{example}

\subsection{Stochastic approach}

Consider the Fréchet class $\BBB_d(\pp)$, i.e., the class of $d$-dimensional Bernoulli pmfs with marginal means $\pp = (p_1,\dots,p_d)$.
The aim of this stochastic approach is to build a non-exchangeable Bernoulli random vector that belongs to a given Fréchet class $\BBB_d(\pp)$, with $0 \leq p_1 \leq \dots \leq p_d \leq 1/2$, and that inherits some dependence property of an exchangeable Bernoulli random vector $\JJ$ with $\text{CMMB}_d(1/2,\nu)$ distribution.
Consider a vector $\boldsymbol{K} = (K_1,\dots,K_{d})$ of $d$ independent Bernoulli random variables, with $\boldsymbol{K} \in \BBB_d(\boldsymbol{\theta})$, where $\boldsymbol{\theta} = (\theta_1 \dots, \theta_d)$ with $\theta_m = 2p_m$, for every $m \in \{1,\dots,d\}$.
Let us define a Bernoulli random vector $\II$ such that $I_m = K_m J_m$, for every $m \in \{1,\dots,d\}$.
By construction, it follows that $\II \in \BBB_d(\pp)$.
To study negative dependence properties of the random vector $\II$, we state the following proposition.

\begin{proposition} \label{prop:SR_AddingIndependence}
    Let $\JJ = (J_1,\dots,J_d)$ be a multivariate Bernoulli random vector that satisfies the SR property, and let K be a Bernoulli random variable independent of $\JJ$ with mean $p_K \in (0,1)$.
    The random vector $\II_1 = (J_1\cdot K, J_2, \dots, J_d)$ satisfies the SR property.
\end{proposition}
\begin{proof}
    The pgf of $\JJ$ can be written as
    \begin{equation*}
        \PPP_{\JJ}(\zz)
        =
        \E[\zz^{\JJ}] 
        =
        q_1\E (\zz^{\JJ} \vert J_1 = 0) + p_1\E (\zz^{\JJ} \vert J_1 = 1)
        =
        \PPP_{\JJ}(\zz) \big\vert_{z_1=0} + z_1 \frac{\partial}{\partial z_1} \PPP_{\JJ}(\zz),
    \end{equation*}
    where $p_1 = 1-q_1 = \Pr(J_1 = 1)$.
    Therefore, the pgf of $\JJ$ is given by
    \begin{equation*}
    \begin{split}
        \PPP_{\II_1}(\zz) 
        &=
        \E[\zz^{\II_1}] 
        =
        \E[\E (\zz^{\II_1} \vert K)]
        =
        p_K \E[\zz^{\II_1} \vert K=1] + q_K \E[\zz^{\II_1} \vert K=0]
        \\
        &=
        p_K \E[\zz^{\JJ}] + q_K \E[z_2^{J_2} \cdots z_d^{J_d}]
        =
        p_K \PPP_{\JJ}(\zz) + q_K \PPP_{\JJ}(\zz) \big\vert_{z_1=1}
        \\
        &=
        p_K \bigg( \PPP_{\JJ}(\zz) \big\vert_{z_1=0} + z_1 \frac{\partial}{\partial z_1} \PPP_{\JJ}(\zz) \bigg)
        +
        q_K \bigg( \PPP_{\JJ}(\zz) \big\vert_{z_1=0} + 1 \cdot \frac{\partial}{\partial z_1} \PPP_{\JJ}(\zz) \bigg)
        \\
        &=
        \PPP_{\JJ}(\zz) \big\vert_{z_1=0} + (q_K + p_Kz_1) \frac{\partial}{\partial z_1} \PPP_{\JJ}(\zz)
        =
        \PPP_{\JJ}(q_K + p_K \cdot z_1,z_2,\dots,z_d),
    \end{split}
    \end{equation*}
    where $q_K = 1-p_K$.
    If $(z_1,\dots,z_d) \in \mathcal{H}^d$, then $(q_K + p_K \cdot z_1,z_2,\dots,z_d) \in \mathcal{H}^d$, and
    \begin{equation*}
        \PPP_{\II_1}(z_1,z_2,\dots,z_d) = \PPP_{\JJ}(q_K + p_K \cdot z_1,z_2,\dots,z_d) \neq 0,
    \end{equation*}
    and $\II_1$ is SR.
\end{proof}
By iteratively applying Proposition \ref{prop:SR_AddingIndependence} to each component of the random vector $\JJ$, we deduce that $\II$ satisfies the SR property whenever $\JJ$ does.
Moreover, the pgf of $\II$ reads
\begin{equation*}
    \mathcal{P}_{\boldsymbol{I}}(\boldsymbol{z}) 
    =
    \mathcal{P}_{\boldsymbol{J}}(\mathcal{P}_{K_1}(z_1), \dots, \mathcal{P}_{K_{d}}(z_{d})) 
    =
    \mathcal{P}_{\boldsymbol{J}}((1-\theta_1 + \theta_1 z_1), \dots, (1-\theta_{d} + \theta_{d} z_{d})).
\end{equation*}

\bibliographystyle{apalike}
\bibliography{reference}

\end{document}